\documentclass[preprint,12pt]{elsarticle}
\usepackage{amssymb}
\usepackage{amsthm}
\usepackage{tikz}
\usepackage{tikz,pgfplots}
\usetikzlibrary{decorations.markings}
\usepackage{amsmath,amssymb}

\usepackage{color}
\usepackage{graphicx}
\usepackage{caption}

\usepackage[margin=2.5cm]{geometry}

\journal{Discrete Applied Mathematics}

\begin{document}

	\newtheorem{theorem}{Theorem}[section]
	\newtheorem{lemma}[theorem]{Lemma}
	\newtheorem{corollary}[theorem]{Corollary}
	\newtheorem{proposition}[theorem]{Proposition}
	\newtheorem{fact}[theorem]{Fact}
	\newtheorem{observation}[theorem]{Observation}
	\newtheorem{claim}[theorem]{Claim}
	\newtheorem{definition}[theorem]{Definition}
	\newtheorem{conjecture}[theorem]{Conjecture}
	\newtheorem{problem}[theorem]{Problem}
	\newtheorem{remark}[theorem]{Remark}
	\newtheorem{example}[theorem]{Example}
	
	\newcommand{\gptwo}{{\rm gp_2}}
	\newcommand{\diss}{{\rm diss}}
	\newcommand{\igp}{{\rm igp}}
	
	\newcommand{\imp}{{\rm imp}}
	\newcommand{\gp}{{\rm gp}}
	\newcommand{\sun}{{\rm Sun}}
	\newcommand{\mono}{{\rm mp}}
	\newcommand{\diam}{{\rm diam}}
	\newcommand{\s}{{\rm s}}
	\newcommand{\TODO}[1]{\textcolor{red}{TODO: #1}}

	\begin{frontmatter}
		
		
		\title{On monophonic position sets in graphs}
		

		
		\author[label1]{Elias John Thomas}
		\ead{eliasjohnkalarickal@gmail.com}
		\author[label2]{Ullas Chandran S. V.}
		\ead{svuc.math@gmail.com}
		\author[label3]{James Tuite}
		\ead{james.tuite@open.ac.uk}
		\author[label4]{Gabriele {Di Stefano}}
		\ead{gabriele.distefano@univaq.it}

		\address[label1]{Department of Mathematics, Mar Ivanios College, University of Kerala, Thiruvananthapuram-695015, Kerala, India}
		
		\address[label2]{Department of Mathematics, Mahatma Gandhi College, University of Kerala, Thiruvananthapuram-695004, Kerala, India}
		\address[label3]{Department of Mathematics and Statistics, Open University, Walton Hall, Milton Keynes, UK}
		\address[label4]{Department of Information Engineering, Computer Science and Mathematics, University of L'Aquila,  Italy}

		\begin{abstract}
			The general position problem in graph theory asks for the largest set $S$ of vertices of a graph $G$ such that no shortest path of $G$ contains more than two vertices of $S$. In this paper we consider a variant of the general position problem called the \emph{monophonic position problem}, obtained by replacing `shortest path' by `induced path'. We prove some basic properties and bounds for the monophonic position number of a graph and determine the monophonic position number of some graph families, including unicyclic graphs, complements of bipartite graphs and split graphs. We show that the monophonic position number of triangle-free graphs is bounded above by the independence number. We present realisation results for the general position number, monophonic position number and monophonic hull number. Finally we discuss the complexity of the monophonic position problem.	\end{abstract}
		
		
	\end{frontmatter}

	\section{Introduction}\label{introduction}
	In 1900 Dudeney, famous for his mathematical puzzles, posed the following question~\cite{dudeney-1917}: what is the largest number of pawns that can be placed on an $n \times n$ chessboard such that no three pawns are on a straight line? This problem was generalised to the setting of graph theory independently at least three times in~\cite{ullas-2016,Korner,manuel-2018a} as follows: a set of vertices $S$ in a graph $G$ is in \emph{general position} if no shortest path of $G$ contains more than two vertices of $S$. The problem then consists of finding the largest set of vertices in general position for a given graph $G$. This has been shown to be an NP-complete problem~\cite{manuel-2018a}. The general position problem has been the subject of intensive research; for some recent developments see~\cite{Ghorbani-2020,Klavzar-2019,manuel-2018b,Neethu-2020,Patkos-2019,Thomas-2020}. 
	
	Some interesting variants of the general position problem have been considered in the literature. In~\cite{KlaKuzPetYer} the authors consider the general position problem using the \emph{Steiner distance} instead of the normal graph distance. In~\cite{KlaRalYer} the authors set a limit on the length of the shortest paths considered. For a fixed integer $d$, they define a set $S$ of vertices of a graph $G$ to be a \emph{general $d$-position set} if for any three vertices $u,v,w \in S$ that lie on a common geodesic $P$, the length of $P$ is greater than $d$; the number of vertices in a largest general $d$-position set is the \emph{general $d$-position number} $\gp _d(G)$ of $G$. The paper~\cite{ThoCha} discusses the largest general position sets that are also independent sets. A further recent variant of the general position problem can be found in~\cite{diStef}, which discusses \emph{mutual visibility sets}; a set $M$ of vertices of a graph $G$ is \emph{mutually visible} if for any $u,v \in M$ there exists at least one shortest $u,v$-path in $G$ that intersects $M$ only in the vertices $u,v$.  In this paper, we consider a variation of the general position problem, which we call the \emph{monophonic position problem}, obtained by replacing `shortest path' by `induced path'.   
	
	We now define the terminology that will be used in this paper. By a graph $G = (V,E)$ we mean a finite, undirected simple graph. The set of neighbours of a vertex $u$ will be written $N(u)$. The \emph{distance} $d(u,v)$ between two vertices $u$ and $v$ in a connected graph $G$ is the length of a shortest $u,v$-path in $G$; any such shortest $u,v$-path is a \emph{geodesic}. A path $P$ in $G$ is \emph{induced} or \emph{monophonic} if $G$ contains no chords between non-adjacent vertices of $P$.
	
	We will denote the subgraph of $G$ induced by a subset $S \subseteq V(G)$ by $G[S]$. A vertex is \emph{simplicial} if its neighbourhood induces a clique; in particular every leaf is simplicial. We denote the number of simplicial vertices and leaves of a graph $G$ by $\s(G)$ and $\ell (G)$ respectively. The \emph{clique number} $\omega (G)$ of $G$ is the number of vertices in a maximum clique in $G$ and the \emph{independence number} $\alpha (G)$ is the number of vertices of a maximum independent set. A subset $S \subseteq V(G)$ is an \emph{independent union of cliques} of $G$ if $G[S]$ is a disjoint union of cliques; the number of vertices in a maximum independent union of cliques will be written as $\alpha ^{\omega }(G)$. A graph $G$ is a \emph{block graph} if every maximal 2-connected component of $G$ is a clique.  
	
	The path of order $\ell $ and length $\ell -1$ will be written as $P_{\ell}$ and the cycle of length $\ell $ as $C_{\ell }$. The \emph{join} $G \vee H$ of two graphs is the graph formed from the disjoint union of $G$ and $H$ by joining every vertex of $G$ to every vertex of $H$ by an edge. Let $G$ and $H$ be graphs where $V(G) = \{v_1 ,\dots, v_n\} $; then the \emph{corona product} $G \odot H$ is obtained from the disjoint union of $G$ and $n$ disjoint copies of $H$, say $H_1,\dots,H_n$, by making the vertex $v_i$ adjacent to every vertex in $H_i$ for $1 \leq i\leq n$. Finally, the \emph{Cartesian product} $G \Box H$, is the graph with vertex set $V(G\Box H) = V(G) \times V(H)$ such that two vertices $(u_1,v_1),(u_2,v_2)$ are adjacent in $G \Box H$ if and only if either $u_1 = u_2$ and $v_1 \sim v_2$ in $H$, or else $v_1 = v_2$ and $u_1 \sim u_2$ in $G$. 
	
	For two distinct vertices $u,v$ of a graph $G$, the \emph{monophonic interval} $K[u,v]$ is the set of all vertices lying on at least one monophonic path connecting $u$ and $v$. The \emph{monophonic closure} of a set $M \subseteq V(G)$ is $K[M]=\bigcup_{_{u,v\in M}}K[u,v].$ If $K[M] = M$, then $M$ is \emph{monophonically convex} or \emph{m-convex}. A smallest m-convex set containing $M$ is an \emph{m-convex hull} of $M$ and is denoted by $[M]_m$.  It is possible to construct the monophonic convex hull $[M]_m $ from the sequence $\{K_ {k}[M]\}$, $k \geq 0$, where $K_ {0}[M]=M$, $K_ {1}[M]=K[M]$ and $K_ {k}[M]=K[K_ {k-1}[M]]$ for $k \geq 2$. From some term onwards, the sequence must be constant; if $r$ is the smallest number such that $K_{r}[M]=K_{r+1}[M]$, then $K_{r}[M] = [M]_m$. A set $M \subseteq V(G)$ is a \emph{monophonic hull set} if $[M]_m = V(G)$. The \emph{monophonic hull number} $h_{m}(G)$ of $G$ is the number of vertices in a smallest monophonic hull set of $G$. A vertex $u$ in $M$ is said to be an \emph{m-interior} vertex of $M$ if $u\in K[M \setminus \{u\}]$ and the set of all m-interior vertices of $M$ is denoted by $M^0$. For any other graph-theoretical terminology we refer to~\cite{BonMur}.   
	
	The plan of this paper is as follows. In Section~\ref{Sec:monophonic position sets}  we introduce the monophonic position number of a graph, determine some simple bounds and discuss the behaviour of the monophonic position number under some graph operations. In Section~\ref{sec:triangle-free} we give a sharp bound for the mp-number of triangle-free graphs and determine the mp-numbers of unicyclic graphs and the join and corona products of graphs. In Section~\ref{sec:split graphs} we find the mp-numbers of split graphs and complements of bipartite graphs. Section~\ref{sec:realisation} provides realisation results for the gp-number, mp-number and monophonic hull number. Finally in Section~\ref{sec:comp} we consider the computational complexity of the monophonic position problem.

	\section{Monophonic position sets in graphs}\label{Sec:monophonic position sets} 
	
	Recall that a set $S$ of vertices in a graph $G$ is a \emph{general position set} if no shortest path in $G$ contains more than two vertices of $S$; by convention a \emph{gp-set} is a largest general position set of $G$. The number of vertices in a largest general position set of $G$ is called the \emph{general position number} of $G$ and is denoted by $\gp(G)$.  We have the following result for the general position number of graphs with diameter two.
	
	\begin{theorem}[\cite{O1}]\label{thm:diameter2}
		If $\diam(G) = 2$, then $\gp(G) = \max\{\omega(G), \eta(G)\},$ where $\eta(G)$ is the maximum order of an induced complete multipartite subgraph of the complement of $G$. 
	\end{theorem}
	We now introduce the following variant of the gp-number.
	\begin{definition}
		A set $M \subseteq V(G)$ is a \emph{monophonic position set} or \emph{mp-set} of $G$ if no three vertices of $M$ lie on a common monophonic path in $G$. The \emph{monophonic position number} or \emph{mp-number} $\mono(G)$ of $G$ is the number of vertices in a largest mp-set of $G$.
	\end{definition}
	For an example of these concepts see Figure~\ref{fig:Petersen}. Observe that every monophonic position set $S$ of a graph $G$ is also in general position; it follows that $\mono(G) \leq \gp(G)$. Any pair of vertices is in monophonic position, so for graphs with order $n \geq 2$ we have $2 \leq \mono(G) \leq n$. It is easily seen that a connected graph satisfies $\mono(G) = n$ if and only if $G \cong K_n$ and the only connected graphs $G$ with $\mono (G) = n-1$ are a) the joins of $K_1$ with a disjoint union of cliques and b) graphs formed from a clique by deleting between one and $n-2$ edges incident to a given vertex. Also the mp-number of the cycle $C_n$ is given by $\mono (C_n) = 2$ for $n \geq 4$.

	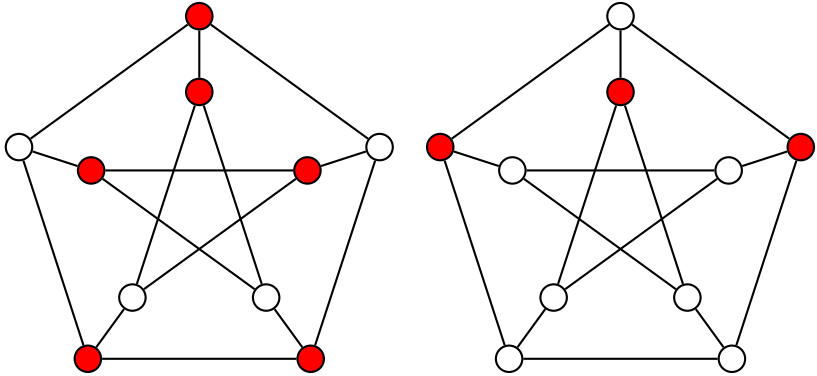
\begin{figure}
		\centering
		\begin{tikzpicture}[x=0.2mm,y=-0.2mm,inner sep=0.2mm,scale=0.7,thick,vertex/.style={circle,draw,minimum size=10}]
			\node at (180,200) [vertex,fill=red] (v1) {};
			\node at (8.8,324.4) [vertex] (v2) {};
			\node at (74.2,525.6) [vertex,fill=red] (v3) {};
			\node at (285.8,525.6) [vertex,fill=red] (v4) {};
			\node at (351.2,324.4) [vertex] (v5) {};
			\node at (180,272) [vertex,fill=red] (v6) {};
			\node at (116.5,467.4) [vertex] (v7) {};
			\node at (282.7,346.6) [vertex,fill=red] (v8) {};
			\node at (77.3,346.6) [vertex,fill=red] (v9) {};
			\node at (243.5,467.4) [vertex] (v10) {};

			\node at (580,200) [vertex] (u1) {};
			\node at (408.8,324.4) [vertex,fill=red] (u2) {};
			\node at (474.2,525.6) [vertex] (u3) {};
			\node at (685.8,525.6) [vertex] (u4) {};
			\node at (751.2,324.4) [vertex,fill=red] (u5) {};
			\node at (580,272) [vertex,fill=red] (u6) {};
			\node at (516.5,467.4) [vertex] (u7) {};
			\node at (682.7,346.6) [vertex] (u8) {};
			\node at (477.3,346.6) [vertex] (u9) {};
			\node at (643.5,467.4) [vertex] (u10) {};
			\path
			(v1) edge (v2)
			(v1) edge (v5)
			(v2) edge (v3)
			(v3) edge (v4)
			(v4) edge (v5)
			
			(v6) edge (v7)
			(v6) edge (v10)
			(v7) edge (v8)
			(v8) edge (v9)
			(v9) edge (v10)
			
			(v1) edge (v6)
			(v2) edge (v9)
			(v3) edge (v7)
			(v4) edge (v10)
			(v5) edge (v8)

			(u1) edge (u2)
			(u1) edge (u5)
			(u2) edge (u3)
			(u3) edge (u4)
			(u4) edge (u5)
			
			(u6) edge (u7)
			(u6) edge (u10)
			(u7) edge (u8)
			(u8) edge (u9)
			(u9) edge (u10)
			
			(u1) edge (u6)
			(u2) edge (u9)
			(u3) edge (u7)
			(u4) edge (u10)
			(u5) edge (u8)
			
			;
		\end{tikzpicture}
		\caption{The Petersen graph with a maximum gp-set (left) and a maximum mp-set (right)}
		\label{fig:Petersen}
	\end{figure}
	
	We begin by noting that for a wide class of graphs the monophonic position number coincides with the general position number. A graph $G$ is \emph{distance-hereditary} if for any connected induced subgraph $H$ of $G$ and vertices $u,v \in H$ we have $d_H(u,v) = d_G(u,v)$, where $d_G(u,v)$ is the distance between $u$ and $v$ in $G$ and $d_H(u,v)$ is the distance between $u$ and $v$ in the subgraph $H$. Distance-hereditary graphs have been characterised by Howorka~\cite{Howorka-77} and many works in the literature are dedicated to them as well as to their generalisations or specialisations (see e.g.~\cite{JDA04,CD06,GDS12,Howorka-81}). In a distance-hereditary graph a path is a geodesic if and only if it is induced; hence the definitions of general position and monophonic position coincide for such graphs.
	
	\begin{observation}\label{lem:dh}
		For any distance-hereditary graph $G$ we have $\mono (G) = \gp(G)$.
	\end{observation}
	
	Distance-hereditary graphs $G$ are not characterised by the relation $\mono(G) = \gp(G)$; for example, we have $\mono(C_5 \odot K_1) = \gp(C_5 \odot K_1) = 5$, but this graph is not distance-hereditary. Since the class of distance-hereditary graphs includes other well studied classes of graphs such as cographs, block graphs, trees and Ptolemaic graphs, results provided for these classes can be seen as a consequence of Observation~\ref{lem:dh}. Hence by Theorem~\ref{thm:diameter2} and the result of~\cite{manuel-2018a} we immediately have the following corollaries.
	
	\begin{corollary}\label{cor:block graphs}
		For any block graph $G$ with $\s(G)$ simplicial vertices, $\mono (G) = \s(G)$. In particular, for any tree $T$ with $\ell (T)$ leaves, we have $\mono(T) = \ell(T)$.
	\end{corollary}
	
	\begin{corollary}\label{cor:complete multipartite graphs}
		For integers $r_1\geq r_2 \geq \dots \geq r_t$ the mp-number and gp-number of the complete multipartite graph $K_{r_1,r_2,\dots ,r_t}$ are given by
		\[ \gp(K_{r_1,r_2,\dots ,r_t}) = \mono(K_{r_1,r_2,\dots ,r_t}) = \max \{ r_1,t\} . \] 
	\end{corollary}
	
	We now provide some bounds on the monophonic position number of a graph. It was shown in~\cite{ullas-2016} that the gp-number of a graph with diameter $D$ is bounded above by $n-D+1$; there is an analogous upper bound for the mp-number of a graph in terms of the length of its longest induced path.
	
	\begin{proposition}\label{n-L+1}
		If $G$ is a graph with order $n$ and the longest induced path in $G$ has length $L$, then then $\mono(G) \leq n-L+1$. This bound is sharp.
	\end{proposition}	
	\begin{proof}
		Let $P$ be a longest monophonic path in $G$ with length $L$ and let $M$ be a maximum mp-set.  The path $P$ can contain at most two vertices of $M$, so that at least $L-1$ vertices of $G$ do not lie in $M$. The bound is sharp for cliques and paths.
	\end{proof}
	In~\cite{manuel-2018a} the authors bound the $\gp $-number of a graph $G$ using the \emph{isometric-path number}, which is the smallest number of geodesics in $G$ such that each vertex of $G$ is contained in exactly one of the geodesics. There is a similar upper bound for the $\mono $-number in terms of the \emph{induced path number} $\rho(G)$, which is defined to be the smallest number of induced paths in $G$ such that each vertex of $G$ is contained in exactly one of the paths. This parameter was introduced in~\cite{ChaMcCSheHosHas}.
	
	\begin{lemma}\label{lem:induced path num upper bound}
		The mp-number of a graph $G$ is bounded above by $\mono(G) \leq 2\rho(G)$.
	\end{lemma}
	\begin{proof}
		Consider a partition of $V(G)$ into $\rho (G)$ induced paths and let $M$ be a largest $\mono $-set. $M$ can intersect each path in the partition in at most two vertices, so that $|M| \leq 2\rho(G)$.
	\end{proof}
	It is shown in~\cite{AkbHorWan} that the induced path number of any connected cubic graph $G$ with order $n \geq 7$ is at most  $\rho(G) \leq \frac{n-1}{3}$. Thus Lemma~\ref{lem:induced path num upper bound} has the following interesting corollary.
	
	\begin{corollary}\label{cor:cubicgraph}
		For any connected cubic graph with order $n \geq 7$, the monophonic position number is at most $\mono (G) \leq \left \lfloor \frac{2(n-1)}{3} \right \rfloor $.
	\end{corollary}
	Corollary~\ref{cor:cubicgraph} is tight; for example the cube has order eight and $\mono $-number four. However, the authors conjecture that the coefficient $\frac{2}{3}$ is not best possible asymptotically.
	
	\begin{conjecture}
		The the largest possible monophonic position number of a cubic graph with order $n$ is $\frac{n}{3}+O(n)$.
	\end{conjecture}
	
	We can also bound $\mono(G)$ above in terms of the number of cut-vertices in $G$.
	
	\begin{lemma}\label{lemma:cut vertex}
		Let $G$ be any connected graph of order $n$ and $c(G)$ cut-vertices. Then $G$ has a maximum monophonic position set that does not contain any cut-vertices. Thus $\mono (G) \leq n-c(G)$.
	\end{lemma}
	\begin{proof}
		Let $M$ be a maximum mp-set of $G$ that contains as few cut-vertices of $G$ as possible. Suppose for a contradiction that $M$ contains a cut-vertex $v$ of $G$. Let $W_1,W_2,\dots, W_k,$ $k \geq 2$, be the components of $G \setminus \{ v\} $.  Then $M$ intersects at most one component, say $W_1$, or else there will be a monophonic path between vertices of $M$ lying in different components that must pass through $v$. Let $W$ be the subgraph of $G$ induced by $W_2 \cup \{ v\} $ and $P$ be a longest path in $W$ with initial vertex $v$. It is easily seen that the terminal vertex $u$ of $P$ is not a cut-vertex of $G$. The set $M' = (M \setminus\{ v\} )\cup \{ u\} $ is also a maximum mp-set of $G$, but contains fewer cut-vertices of $G$ than $M$, contradicting the definition of $M$.
	\end{proof}
	Lemma~\ref{lemma:cut vertex} is sharp for block graphs. We now give a lower bound for $\mono (G)$ in terms of the number of simplicial vertices in $G$.
	\begin{lemma}\label{lemma:simplicial vertex}
		For any graph $G$ we have $\mono(G) \geq \s(G)$.
	\end{lemma}
	\begin{proof}
		Let $S$ be the set of simplicial vertices of $G$.  Suppose that $u,v,w \in S$ and that $P$ is a monophonic $u,v$-path containing $w$ (in particular, this requires $\s(G) \geq 3$). The vertex $w$ has two neighbours $w'$ and $w''$ on $P$; however, by definition, we must have $w' \sim w''$, so that $P$ is not induced, a contradiction.
	\end{proof}
	Corollary~\ref{cor:block graphs} shows that the lower bound in Lemma~\ref{lemma:simplicial vertex} is also sharp for block graphs. Finally we present two results that show the effect on monophonic position sets of adding a leaf to a graph. Recall that for any set $M \subset V(G)$ a vertex $u$ of $G$ is in the interior $M^0$ of $M$ if and only if there are vertices $x,y \in M \setminus \{ u\} $ such that $u$ lies on a monophonic $x,y$-path.
	
	\begin{lemma}\label{lemma:Trees}
		Let $ G^\prime $ be a graph obtained from a graph $G$ by adding a pendant edge $uv$ at a vertex $v$ of $G$. Then $\mono(G) \leq \mono(G^\prime) \leq \mono(G) + 1$. Moreover, $\mono(G^\prime) = \mono(G) + 1$ if and only if there exists a maximum mp-set $M$ of $G$ with $v\not \in M$ such that $\big{(} M \cup \{ v \} \big{)}^0 = \{ v \} $.
	\end{lemma}
	\begin{proof}
		Since every mp-set of $G$ is also an mp-set of $G^\prime $, it follows that $\mono(G) \leq \mono(G^\prime)$; conversely, if $M^\prime$ is a maximum mp-set of $ G^\prime $, then $M^\prime \setminus \{u \}$ is an mp-set of $G$, so that $|M^\prime \setminus \{u \}|\leq \mono(G)$ and hence $\mono(G^\prime) = |M^\prime| \leq \mono(G) +1$.
		
		Suppose that $\mono(G^\prime) = \mono(G) + 1$ with a largest mp-set $M^\prime $; then $\mono(G^\prime)\geq 3$ and $u \in M^\prime $. Since $|M^\prime| \geq 3$, it follows that $v \not \in M^\prime $, for otherwise $v$ would lie on a monophonic path from $u$ to another member of $M^\prime $ in $G^\prime $. Let $M = (M^\prime \cup \{ v \}) \setminus \{ u \} $.  $M$ is too large to be an mp-set of $G$, so there must exist three vertices $x,y,z$ of $M$ that lie on a common monophonic path in $G$. Since $M^\prime \setminus \{ u\} = M \setminus \{ v\} $ is an mp-set of $G$, one of these vertices, say $x$, must be $v$. If $y \in K[v,z]$ in $G$, then we would have $y \in K[u,z]$ in $G^\prime $, contradicting the fact that $M^\prime $ is an mp-set of $G^\prime $. Therefore $v \in K[y,z]$ and so $v$ is the only vertex in $M$ such that $v\in M^0.$ Thus $M \setminus \{ v\} $ is the required set.   
		
		Conversely, assume that there exists a maximum mp-set $M$ of $G$ with $v\not \in M$ and $ \big{(} M \cup \{ v \} \big{)}^0 = \{ v \} $. We claim that $M^\prime = M \cup \{u \} $ is an mp-set of $ G^\prime $. Suppose that there exist $x, y, z \in M^\prime $ such that $x \in K[y, z]$ in $G^\prime $. As $u$ is a leaf, $x \not = u$, so we can set $z = u$. Then $x \in K[y, v]$ and hence $ x \in \big{(} M \cup \{ v \} \big{)}^0,$ which is impossible. Thus $ M^\prime $ is an mp-set in $ G^\prime $  with $|M^\prime| = \mono(G) + 1$. 
	\end{proof}
	
	\begin{proposition}\label{proposition:Trees}
		If $G^\prime $ is a graph obtained from $G$ by adding a pendant vertex $u$ to a simplicial vertex $v$ in $G$, then $\mono(G) = \mono(G^\prime)$.
	\end{proposition}	
	\begin{proof}
		By Lemma~\ref{lemma:Trees} $\mono(G) \leq \mono(G^\prime ) \leq \mono(G)+1$ and if $\mono(G^\prime ) = \mono(G)+1$, then there exists an mp-set $M$ of $G$ with $|M| = \mono(G)$ such that $v \not \in M$ and $\big{(} M \cup \{ v \} \big{)}^0 = \{ v \} $.  However, as a simplicial vertex $v$ cannot be an interior vertex of any monophonic path in $G$, we have $\mono(G) = \mono(G^\prime )$.   
	\end{proof}

	\section{Monophonic position in graph families}\label{sec:triangle-free} 
	
	In this section we discuss the monophonic position numbers of some common graph families. In Theorem~\ref{triangle-free} we give a sharp bound for the mp-numbers of triangle-free graphs. We then give exact expressions for the mp-numbers of unicyclic graphs and graphs formed as the join or corona product of two graphs.

	\begin{lemma}\label{lemma:Bipartite graphs}
		Let $G$ be a connected graph and $M\subseteq V(G)$ be an mp-set. Then $G[M]$ is a disjoint union of $k$ cliques $G[M]=\bigcup_{i=1}^{k} W_i $. If $k \geq 2$, then for $1 \leq i \leq k$ any two vertices of $W_i$ have a common neighbour in $G \setminus M$.
	\end{lemma}
	\begin{proof}
		Let $W_1, W_2,\dots,W_k$ be the components of $G[M]$. If some $W_i$ is not a clique, then it would contain an induced path of length two, which is impossible; hence $M$ is an independent union of cliques. Let $k \geq 2$ and suppose for a contradiction that there is a component, say $W_1$, with $u,v\in V(W_1)$ such that $u$ and $v$ have no common neighbour in $G \setminus M$. Let $w$ be any vertex in $W_2$ and let $P$ be a $u,w$-monophonic path $u,u_1,u_2, \dots, u_{\ell} = w$ in $G$. Then $\ell \geq 2$. Since $M$ is an mp-set, it follows that $P$ together with the edge $uv$ is not a monophonic path in $G$. This shows that $v$ must be adjacent to $u_j$ for some $j$ with $1 \leq j \leq \ell-1$. It follows from the choice of $u$ and $v$ that $j \geq 2$; let $j$ be the largest suffix for which $v$ is adjacent to $u_j$. Then the $u_j,w$-subpath of $P$ together with the 2-path $u_j,v,u$ forms a monophonic path containing three points of $M$; hence $u$ and $v$ must have a common neighbour in $G \setminus M$. 
	\end{proof}

	\begin{theorem}\label{triangle-free}
		The mp-number of a connected triangle-free graph $G$ with order $n \geq 3$ satisfies $\mono(G) \leq \alpha(G)$. Moreover, if the length of any monophonic path is at most three, then $\mono(G) = \alpha(G)$.
	\end{theorem}
	\begin{proof}
		Let $M$ be a maximum mp-set in a triangle-free graph $G$. By Lemma~\ref{lemma:Bipartite graphs} $M$ is an independent union of cliques $W_1, W_2,\dots,W_k$. If $k = 1$, then $G[M]$ is a clique and so $\mono(G) = |M| \leq \omega (G) = 2 \leq \alpha (G),$ so assume that $k \geq 2.$ If any component $W_i$ of $G[M]$ contains distinct vertices $u,v$, then by Lemma~\ref{lemma:Bipartite graphs} $u$ and $v$ have a common neighbour in $G \setminus M$, so that $G$ would contain a triangle. It follows that each component of $G[M]$ consists of a single vertex, so that $M$ is an independent set of $G$ and $\mono(G) \leq \alpha (G)$.  
		
		Let $T$ be a maximum independent set of $G$. If three vertices of $T$ lie on a common monophonic path $P$, then the length of $P$ must be at least four. Therefore if the longest monophonic path of $G$ has length at most three, then we have equality in the bound and any maximum independent set of $G$ will be a maximum mp-set.
	\end{proof}
	The bound in Theorem~\ref{triangle-free} is sharp; it is met by the caterpillar formed by adjoining one leaf to every internal vertex of a path, complete bipartite graphs $K_{r,r}$ and the corona product $C_n \odot K_1$ for $n \geq 4$. Hence this bound can be achieved by graphs with arbitrarily large diameter, minimum degree and girth. 
	
	Theorem~\ref{triangle-free} can be used to find the monophonic position numbers of some graphs with large girth. This problem is particularly interesting for the cage graphs. The unique cubic cages with girths five, six and seven are the Petersen graph, the Heawood graph and the McGee graph respectively. We omit the lengthy case argument used to prove the following result, which was checked computationally by Erskine~\cite{Erskine}.
	\begin{theorem}
		The monophonic position numbers of the Petersen graph, the Heawood graph and the McGee graph are three, three and two respectively.
	\end{theorem}
	This motivates the following conjecture.
	\begin{conjecture}
		For sufficiently large $g$, the monophonic position number of a $(d,g)$-cage $G$ satisfies $\mono (G) < d$.
	\end{conjecture}
	
	We now determine the mp-numbers of unicyclic graphs; as we have seen, such graphs can meet the upper bound in Theorem~\ref{triangle-free}. We will identify the vertex set of the unique cycle $C$ of a unicyclic graph $G$ with $\mathbb {Z}_s$, where the length of the cycle is $s \geq 3$, where $i \sim i+1$ for $0 \leq i \leq s-1 \pmod s$. We denote by $R$ the set of vertices of $C$ that have degree $\geq 3$. We will call any vertex in $R$ a \emph{root} and if $i \in R$ we write $T_i$ for the tree in $G \setminus \{ i-1,i+1\} $ that contains $i$. The leaf number of the tree $T_i$ is $\ell (T_i)$ and we write $\ell ^{\prime }(T_i)$ for the number of leaves of $T_i$ that are also leaves of $G$. We have already dealt with the case that $G$ is a cycle ($\mono (K_3) = 3$, $\mono (C_s) = 2$ for $s \geq 4$), so we assume that $|R| > 0$.  
	
	\begin{theorem}
		Let $G$ be a unicyclic graph that is not a cycle. Then $$\mono(G) =
		{
			\begin{cases}
				\ell(G)+2, & \text{ if } r = 1,\\
				\ell(G)+1, & \text{ if } r = 2, R = \{ i,j\} \text{ and either } T_i \text{ or } T_j \text{ is a path},\\
				\ell(G)+1, & \text{ if } r = 2 \text{ and } R = \{ i,i+1\} \text{ for some } 0 \leq i \leq s-1,\\
				\ell(G), & \text{ otherwise.}
		\end{cases}}$$
	\end{theorem}
	\begin{proof}
		Let $M$ be a maximum mp-set of $G$; by Lemma~\ref{lemma:cut vertex} we can assume that $M$ contains no cut-vertices, so that $M \cap T_i$ consists of leaves for any tree $T_i$ (in particular $M \cap (C \setminus R) = \emptyset $). It follows from Corollary~\ref{cor:block graphs} and the mp-numbers of the cycles that $\ell (G) \leq \mono (G) \leq \ell (G)+2$. If $|R| = 1$, then the leaves of the unique tree $T_i$ and the vertices $\{ i-1,i+1\} $ form an mp-set, so that the upper bound is achieved, so we can take $|R| \geq 2$. 
		
		Suppose that $\mono (G) = \ell (G)+2$. Then $M$ must contains $\ell ^{\prime }(T_i)$ vertices in each tree $T_i$ for $i \in R$, as well as two vertices of $C \setminus R$. As we are assuming that $|R| \geq 2$, this is not possible, for either there is a monophonic path from a vertex of $M$ in a tree $T_i$ through a section of $C$ containing two vertices of $M$, or else a monophonic path from a vertex of $M$ in a tree $T_i$ to a vertex of $M$ in a tree $T_j$, $i \not = j$, through a vertex of $M$ on $C$. Thus for $|R| \geq 2$ we have $\mono (G) \leq \ell (G)+1$. 
		
		Suppose that $\mono (G) = \ell (G)+1$. Then either $|M \cap (C \setminus R)| = 1$ and $|M \cap T_i| = \ell ^{\prime }(T_i)$ for each $i \in R$, or else $|C \setminus R| = 2$ and there is one tree $T_i$ with $|M \cap T_i| = \ell ^{\prime }(T_i) -1$. If $|M \cap (C \setminus R)| = 2$, then there is a unique tree $T_i$ that contains a vertex of $M$ and the vertices of $M \cap (C \setminus R)$ are $\{ i-1,i+1\} $, so that, since no other tree contains a vertex of $M$, we must have $|R| = 2$ and the other tree $T_j$ is a path, in which case equality holds in $\mono (G) = \ell(G) +1$. Hence assume that every tree $T_i$ has $|M \cap T_i| = \ell ^{\prime }(T_i)$ and $|M \cap (C \setminus R) = 1$. If there are two trees $T_i, T_j$ such that $i \not \sim j$ on $C$, then there would be a monophonic path from $M \cap T_i$ to $M \cap T_j$ through the vertex of $M$ on $C \setminus R$, so either $|R| = 2$ or $s = |R| = 3$; in the latter case trivially $\mono (G) = \ell (G)$. If $R = \{ i,i+1\} $, then the leaves of $T_i$ and $T_{i+1}$ together with the vertex $i-1$ of $C$ constitute an $\mono $-set of $G$, so that we have $\mono (G) = \ell (G)+1$ in this case; otherwise $\mono (G) = \ell (G)$.
	\end{proof}

	In~\cite{Ghorbani-2020} the authors determined the gp-number of the join and corona product of graphs.
	
	\begin{theorem}[\cite{Ghorbani-2020}]
		For graphs $G$ and $H$ the general position number of the join $G \vee H$ and the corona product $G \odot H$ are given by 
		\[ \gp(G \vee H) = \max\{\omega (G) + \omega (H), \alpha ^{\omega }(G), \alpha ^{\omega }(H)\}\]
		and if $G$ is connected and has order $n(G)$, then
		\[ \gp(G \odot H) =n(G) \alpha ^{\omega }(H).\]	
	\end{theorem}  
	We now determine the corresponding relation for the monophonic position numbers.
	
	\begin{proposition}\label{theorem: Corona of graphs}
		Let $G$ be a connected graph with order $n(G)$ and let $H$ be any graph. Then $\mono(G \odot H) = n(G)\mono(H)$.
	\end{proposition}
	\begin{proof}
		Let $V(G) = \{v_1 ,\dots, v_n\} $ and let $H_1,\dots,H_n$ be the corresponding copies of $H$ in $G \odot H$. Let $M$ be a maximum mp-set of $G \odot H$. By Lemma~\ref{lemma:cut vertex} we can assume that $M$ does not contain any cut-vertices of $G \odot H$. Observe that every vertex of $G$ is a cut-vertex in $G \odot H$, so that we can take $M$ to lie entirely in $H_1 \cup H_2 \cup \dots \cup H_n$.  It is easily seen that each set $M \cap V(H_i)$ must be in monophonic position; it follows that $|M| \leq n(G)\mono(H)$. Conversely, if $S$ is a maximum mp-set of $H$ and for $1 \leq i \leq n$ the corresponding set in $H_i$ is $S_i$, then $S_1 \cup S_2 \cup \dots \cup S_n$ is in monophonic position. Therefore $\mono(G \odot H) = n(G)\mono(H)$. 
	\end{proof}
	
	\begin{proposition}\label{lemma:join}
		The monophonic position number of the join $G \vee H$ of graphs $G$ and $H$ is related to the monophonic position numbers of $G$ and $H$ by
		\[ \mono(G \vee H) = \max\{\omega (G) + \omega (H), \mono(G), \mono(H) \} .\]
	\end{proposition}  
	\begin{proof}
		It is evident that an mp-set in $G$ is an mp-set in $G \vee H$, so that $\mono(G \vee H) \geq \mono(G)$ and likewise $\mono(G \vee H) \geq \mono(H)$. Also the union of a clique in $G$ and a clique in $H$ is a clique in $G \vee H$ and is hence an mp-set in $G \vee H$, so it follows that $\mono (G \vee H) \geq \max\{\omega (G) + \omega (H), \mono(G), \mono(H)\} $. For the opposite direction, let $M$ be a maximum mp-set in $G \vee H$ and set $M_1 = M \cap V(G)$ and $M_2 = M \cap V(H)$.  If $M_1 = \emptyset $ or $M_2 = \emptyset $, then trivially $\mono(G \vee H) \leq \max \{ \mono(G),\mono(H)\} $, so suppose that $M_1$ and $M_2$ are both non-empty. Suppose that there exist distinct $x_1,x_2 \in M_1$ such that $x_1 \not \sim x_2$; in this case if $y \in M_2$, then $x_1,y,x_2$ would be a monophonic path in $G \vee H$, a contradiction. Therefore $M_1$ and $M_2$ must both induce cliques in $G$ and $H$ respectively, so that $\mono(G \vee H) \leq \omega (G) + \omega (H)$. 
	\end{proof}

	\section{Complements of bipartite graphs and split graphs}\label{sec:split graphs}
	
	\emph{Cographs} are induced $P_4$-free graphs. As cographs are distance-hereditary, the mp- and gp-numbers of these graphs are equal by Observation~\ref{lem:dh}. It is therefore of interest to study the mp-numbers of induced $P_5$-free graphs; in this class equality between the $\gp $- and $\mono $-numbers does not hold in general, as shown by the cycle $C_5$, for which $\mono (C_5) = 2 < 3 = \gp (C_5)$. We therefore consider two well-known classes of induced $P_5$-free graphs, namely the complement of connected bipartite graphs and split graphs.  Since the complement of the complete bipartite graph $K_{m,n}$ is disconnected, clearly $\mono(\overline K_{m,n}) = m+n$.
	
	Let $G$ be a connected bipartite graph with bipartition $(A,B)$ and let $S \subseteq V(G)$. Fix $S_A = S\cap A$ and $S_B = S\cap B$. A set $X \subseteq A$ (or $B$) is \emph{uniform} in $G$ if $N(u) = N(v)$ for all $u, v \in X$; we call $S \subseteq V(G)$ a \emph{uniform set} if both $S_A$ and $S_B$ are uniform in $G$. Let $\psi(G)$ denote the number of vertices in a largest uniform set in $G$. 
	
	\begin{theorem}\label{theorem: Complement of bipartite graphs}
		If $G$ is a connected bipartite graph with bipartition $(A,B)$, then $\mono(\overline G) =  \max\{ \alpha(G),\psi(G) \}.$
	\end{theorem}
	\begin{proof}
		Let $S$ be a maximum $\mono $-set in $\overline G$. Then both $S_A$ and $S_B$ are cliques in $\overline G$. If $S_A = \emptyset$ or $S_B = \emptyset $, then it is clear that $|S| \leq \omega (\overline G) =\alpha(G)$, so assume that $S_A \not = \emptyset $ and $S_B \not = \emptyset $. If there is an edge between $S_A$ and $S_B$ in $\overline G$, then each vertex of $S_A$ must be adjacent to all vertices of $S_B$, so that $S$ induces a clique in $\overline G$ and again $|S| \leq \alpha(G)$. Hence we can assume that there is no edge between $S_A$ and $S_B$ in $\overline G$; thus $S$ induces a complete bipartite subgraph in $G$. 
		
		Next we claim that $S$ is uniform. Assume to the contrary that there exist vertices $u,v \in S_A$ and $w \in B$ such that $w$ is adjacent to $v$ but not to $u$ in $G$. Then $w \notin S_B$. Choose $x \in S_B $ arbitrarily. Then the path $v,u,w,x$ is a $v,x$-monophonic path in $\overline G$ containing the vertex $u$, a contradiction to the fact that $S$ is an mp-set of $\overline G$. Thus $S$ must be uniform in $G$ and so $|S|\leq \psi(G)$. Therefore we have $\mono (\overline G) \leq \max \{ \alpha (G),\psi (G)\} $.
		
		On the other hand, it is clear that $\mono(\overline G) \geq \omega (\overline G) =\alpha(G)$. We now show that every maximum uniform set $S \subseteq V(G)$ gives an $\mono $-set in $\overline G$. Both $S_A$ and $S_B$ are cliques in $\overline G$. If there is no edge between $S_A$ and $S_B$ in $G$, then $S$ induces a clique in $\overline G$, so that $S$ is an $\mono $-set and $\mono(\overline G) \geq |S| = \psi(G)$. Hence we can assume that there is an edge between $S_A$ and $S_B$ in $G$; as $S$ is uniform, $G$ contains all edges between $S_A$ and $S_B$, so that there are no edges between $S_A$ and $S_B$ in $\overline G$. Also by uniformity, there are sets $X \subseteq A \setminus S_A$ and $Y \subseteq B \setminus S_B$ such that $N(x) = Y$ for each $x \in S_A$ and $N(y) = X$ for each $y \in S_B$; from this it is simple to see that $S$ is in monophonic position in $\overline G$, so that $\mono(\overline G) \geq |S|= \psi(G)$, concluding the proof.
	\end{proof}
	
	Let $T$ be a tree with order $n \geq 2$ and suppose that $S$ is a uniform set of $T$ with $|S| > \alpha (T)$. Suppose that $|S_A| \geq 2$; then each vertex in $S_A$ is a leaf, for if $d(u) \geq 2$ for some $u \in S_A$, then for any $v \in S_A \setminus \{ u\} $ the vertices $u$ and $v$ would have at least two common neighbours, which is impossible. As the same reasoning applies to $S_B$ and $\alpha (T) \geq \ell (T)$, we conclude that $S_A$ is a set of leaves and $|S_B| = 1$, say $S_B = \{ w\} $. If $S_A \cup S_B = V(T)$, then $T$ is a star, $V(T)$ is a uniform set and $\mono (T) = n(T)$; otherwise, considering the endpoint of a longest path to a vertex of $V(T) \setminus S$, we see that $T$ contains a leaf that does not belong to $S$, so that after all $|S| \leq \alpha (T)$. This proves the following corollary.
	\begin{corollary}
		If $T$ is a tree with order $n \geq 2$, then $$\mono(\overline T ) =
		{
			\begin{cases}
				n,&\text{ if } T \text{ is a star}, \\	
				
				\alpha(T), & \text{ otherwise.}
		\end{cases}}$$
	\end{corollary}
	
	A similar argument yields the following results.
	\begin{corollary}
		If $n,m \geq 2,$ then $$\mono(\overline {P_n \Box P_m}) =
		{
			\begin{cases}
				4,&\text{ if } n = m = 2, \\	
				
				\alpha( {P_n \Box P_m}) = \lceil \frac{n}{2}\rceil \lceil \frac{m}{2}\rceil + \lfloor \frac{n}{2}\rfloor\lfloor \frac{m}{2}\rfloor, & \text{ otherwise.}
		\end{cases}}$$
		
	\end{corollary}
	\begin{corollary}
		If $k\geq 3$, then $\mono(\overline{Q_k}) = 2^{k-1}$, where $Q_k$ is the hypercube with order $2^k$.  
	\end{corollary}

	A graph is a \emph{split graph} if the vertex set can be partitioned into a clique $C$ and an independent set $I$.  If there is a vertex $v$ that is adjacent to every vertex of $C\setminus \{ v\} $ and non-adjacent with every vertex of $I\setminus \{ v\} $ we will say that $v$ is \emph{divided}. Obviously there cannot be divided vertices in both $C$ and $I$. If there are no divided vertices, then $\omega (G) = |C|$ and $\alpha (G) = |I|$. If $G$ contains a divided vertex $v$, then by moving $v$ from $C$ into $I$ if necessary, we can assume that any divided vertices lie in $I$ and $\omega (G) = |C|+1$ and $\alpha (G) = |I|$. We define a \emph{separated subgraph} $(C',I')$ of the split graph $G= (C,I)$ as follows.
	
	\begin{definition}
		Let $C' \subseteq C$ and $I' \subseteq I$. Then $G' = (C',I')$ is a \emph{separated subgraph} of the split graph $G = (C,I)$ if and only if either
		\begin{itemize}
			\item $($Type A$)$ there is no edge from $C'$ to $I'$ in $G$, or
			\item $($Type B$)$ there is a vertex $v \in I'$ that is adjacent to every vertex of $C'$, there are no edges from $I' \setminus \{ v\} $ to $C'$ and for any $w \in I' \setminus \{v\} $ we have $N(w) \subset N(v)$.
		\end{itemize}
		Let $\phi (G)$ denote the order of a largest separated subgraph of $G$. 
	\end{definition}
	Trivially the clique induced by $C$ is separated, as is the independent set induced by $I$. If there is a divided vertex $v$ in $I$, then $(C,\{ v\})$ is separated. This shows that we always have $\phi (G) \geq \max \{ \omega (G),\alpha (G)\} $.

	\begin{theorem}
		Let $G = (C,I)$ be a connected split graph. Then $\mono(G) = \phi (G)$.
	\end{theorem}
	\begin{proof}
		Let $(C',I')$ be a separated subgraph of $G$. Suppose for a contradiction that $P$ is a monophonic path with endpoints $w_1,w_3 \in C' \cup I'$ that passes through a vertex $w_2 \in (C' \cup I') \setminus \{w_1,w_3\} $. We cannot have $w_1,w_3 \in C'$, as $w_1 \sim w_3$. Also $w_2$ cannot lie in $I'$, as the neighbours of $w_2$ on $P$ would both be in $C$ and so would be adjacent. Suppose that $w_1,w_3 \in I',w_2 \in C'$. If $(C',I')$ is a separated subgraph of Type A, then the monophonic path $P$ would have to include at least three vertices from $C$, which is impossible. Hence $(C',I')$ is a separated subgraph of Type B. Now, if $w_1 = v$, then $P$ cannot be monophonic, as we would have $N(w_3) \subset N(w_1)$, but $w_3 \not \sim w_2$. Hence $w_1,w_2 \in I' \setminus \{ v\}$; however, in this case $P$ would contain at least three vertices from $C$, which is again impossible. Finally suppose that $w_1,w_2 \in C',w_3 \in I'$. Since $P$ has exactly two vertices from $C$, it follows that $(C',I')$ is Type $B$ and $w_3 = v$; this is a contradiction, as we would have $v \sim w_1$. Hence we conclude that any separated subgraph is in monophonic position.
		
		Conversely, we now show that any subgraph $(C',I')$ that is in monophonic position must be separated. If $(C',I')$ is not Type A, then there is a vertex $v \in I'$ with an edge to a vertex $u \in C'$. Suppose that there is a vertex $u' \in C' \setminus N(v)$; then $v,u,u'$ would be a monophonic path in $(C',I')$, a contradiction. Hence $v$ is adjacent to every vertex of $C'$. If there is a vertex $v' \in I' \setminus \{ v\}$ that is adjacent to a vertex $x$ in $C'$, then $v,x,v'$ would be an induced path in $(C',I')$, so there are no edges from $C'$ to $I' \setminus \{ v\} $. Finally suppose that there is an edge $u' \sim v'$, where $u' \in C \setminus C'$, $v' \in I' \setminus \{ v\}$, such that $v \not \sim u'$; in this case $v,u,u',v'$ would be an induced path containing three vertices of $(C',I')$; we conclude that $N(v') \subset N(v)$ for any $v' \in I' \setminus \{ v\} $. Hence $(C',I')$ is a separated subgraph. 
	\end{proof}
	
	We noted before that for any split graph $G$ we have $\mono(G) = \phi(G) \geq \max \{ \omega (G),\alpha (G)\} $. We now characterise the cases in which equality holds. For $X \subseteq V(C)$ write $N_I(X) = I \cap (\bigcup _{x \in X}N(x))$. By a matching between $C$ and $I$ in $G$, we will mean a matching such that no edge of the matching has both endpoints in $C$. 
	\begin{theorem}\label{lower bound split graph} 
		If $G =(C,I)$ is a connected split graph, then if there are no divided vertices in $G$ we have $\mono(G) = \phi (G) =  \max\{\omega(G), \alpha(G)\} $ if and only if there exists a matching in $G$ between $C$ and $I$ that saturates either $C$ or $I$. The same conclusion holds if $G$ has a divided vertex, unless $G$ is formed from a clique $A$ of order $r+2$ and an independent set $B$ of order $> r$ by adding a set $E'$ of edges between $A'$ and $I$, where $A' \subset A$, $|A'| = r$, such that $E'$ contains a matching saturating $A'$, in which case $\mono (G) = \alpha (G)+1 > \omega (G)$.   
	\end{theorem}
	\begin{proof}
		First suppose that there is no matching between $C$ and $I$ that saturates $C$ or $I$; we will show that $\phi (G) > \max \{ \omega (G), \alpha (G)\} $. By Hall's Theorem~\cite{Hall}, as there is no matching between $C$ and $I$ that saturates $C$, there is a subset $K \subseteq C$ such that $|N_I(K)| <|K|$. Then $(K,I \setminus N_I(K))$ is a separated subgraph with order $> |I| = \alpha (G)$. Similarly, considering matchings from $I$, we see that there is subset $J \subseteq I$ such that $|N(J)| < J$, implying that $(C \setminus N(J),J)$ is a separated subgraph with order $|(C \setminus N(J),J)| > |C|$. If $G$ has no divided vertices, then $|C| = \omega (G)$ and we are done, so suppose that $v$ is a divided vertex in $I$. If $v \in J$, then $N(J) = C$, so we have $|C| = |N(J)| < |J| \leq |I| \leq \alpha (G)$ and hence $\omega (G) = |C|+1 \leq \alpha (G)$. However, the separated set $(K,I \setminus N_I(K))$ constructed in the previous discussion has order $> \alpha (G)$ and hence suffices. Otherwise, we can move $v$ from $I \setminus J$ to $J$ to obtain a separated subset of Type B with order $\geq |C|+2 > \omega (G)$. 
		
		Now suppose that there is a matching $M$ between $C$ and $I$ that saturates $C$ or $I$. If $C' \cup I'$ induces a clique or an independent set, then the result follows immediately, so we assume that both $C'$ and $I'$ are non-empty. If $(C',I')$ is Type A, then either the vertices of $I'$ are matched to a subset of $C\setminus C'$ by $M$, in which case $|(C',I')| \leq |C| \leq \omega (G)$, or else the vertices of $C'$ are matched with a subset of $I \setminus I'$, in which case $|(C',I')| \leq \alpha (G)$. 
		
		Thus assume that $(C',I')$ is Type B. Firstly consider the case that $M$ saturates $I$. Then by the previous argument $M$ must contain an edge from $v \in I'$ to $C'$ and the remaining vertices of $I \setminus \{ v\} $ are matched with a subset of $C\setminus C'$. If $C \setminus (C' \cup N(I' \setminus \{ v\} )) \not = \emptyset $, then $|C' \cup I'| \leq |C| \leq \omega (G)$, so we can assume that $N(I' \setminus \{v\}) = C \setminus C'$. By definition of a separated set, $v$ must then be adjacent to every vertex of $C$, so that $v$ is a divided vertex and $|C' \cup I'| \leq |C|+1 = \omega (G)$. 
		
		Now suppose that $M$ saturates $C$; similarly to the previous case, we can assume that $M$ includes an edge from a vertex $u \in C'$ to the unique vertex $v \in I'$ that is adjacent to every vertex of $C$, and that $M$ contains a perfect matching between $C' \setminus \{ u\} $ and $I \setminus I'$, from which it follows that $\mono (G) \leq \alpha (G)+1$. But now it follows that each vertex of $C \setminus C'$ is matched by $M$ to a vertex of $I' \setminus \{ v\} $, so that $C \setminus C' = N(I' \setminus \{ v\} )$ and, by definition of separated set, $v$ is adjacent to every vertex of $C \setminus C'$ and $v$ is divided. Hence if there are no divided vertices we have equality in the lower bound. 
		
		Let $H$ be any bipartite graph with bipartition $(A,B)$ such that a) $|B| \geq |A|+1$ and b) there is a matching in $H$ that saturates $A$. Form a split graph $H' = (C,I)$ as follows: add two new vertices $x,y$ and add every edge between vertices in $\{ x,y\} \cup A$, then set $C = A \cup \{ x\} , I = B \cup \{ y\} $. The original matching from $A$ to $B$ in $H$ plus the edge $x \sim y$ gives the required matching. Also the set $B \cup \{ x,y\} $ is a largest mp-set, so that $\mono (H') = \alpha (H')+1 > \omega (H')$; in fact the previous discussion shows that a split graph $G$ has $\mono (G) = \alpha (G)+1 > \omega (G)$ if and only if it belongs to this family.
	\end{proof}

	\section{Realisation results}\label{sec:realisation}

	In this paper we have presented some properties of the monophonic position number of graphs.  In some respects this graph parameter behaves like the more widely studied general position number. It is therefore of interest to ask whether there is a relation between the two numbers other than the trivial inequality $\mono(G) \leq \gp(G)$? We will show that these two parameters are independent by proving the following realisation result: for any pair $a,b \in \mathbb{N}$ such that $2 \leq a \leq b$ there exists a graph $G$ with $\mono(G) = a$ and $\gp(G) = b$.  Firstly observe that if $a = b$, then trivially the complete graph $K_a$ has the required properties, so in the following we will assume that $a < b$.
	
	\begin{theorem}\label{half wheel}
		For any $2 \leq a \leq b$ there exists a graph $G$ with $\mono(G) = a$ and $\gp(G) = b$.
	\end{theorem}  
	\begin{proof}
		We begin with the case $a = 2$. We have $\mono(C_5) = 2, \gp(C_5) = 3$, so we can assume that $b \geq 4$. We define the \emph{half-wheel} graph $H_r$ for $r \geq 2$ as follows. Take a cycle $C_{2r}$ of length $2r$ and label the vertices by the elements of $\mathbb{Z}_{2r}$ in the natural manner. Add an extra vertex $x$ and join $x$ to all vertices in $C_{2r}$ with even labels. An example is shown in Figure~\ref{fig:H4}.

		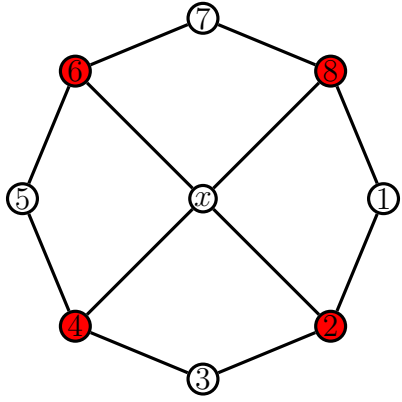
\begin{figure}
			\centering
			\begin{tikzpicture}[x=0.4mm,y=-0.4mm,inner sep=0.2mm,scale=0.6,very thick,vertex/.style={circle,draw,minimum size=10,fill=white}]
				\node at (100,0) [vertex] (v1) {$1$};
				\node at (70.71,70.71) [vertex,fill=red] (v2) {$2$};
				\node at (0,100) [vertex] (v3) {$3$};
				\node at (-70.71,70.71) [vertex,fill=red] (v4) {$4$};
				\node at (-100,0) [vertex] (v5) {$5$};
				\node at (-70.71,-70.71) [vertex,fill=red] (v6) {$6$};
				\node at (0,-100) [vertex] (v7) {$7$};
				\node at (70.71,-70.71) [vertex,fill=red] (v8) {$8$};
				\node at (0,0) [vertex] (x) {$x$};			
				
				\path
				(v1) edge (v2)
				(v2) edge (v3)		
				(v3) edge (v4)		
				(v4) edge (v5)
				(v5) edge (v6)
				(v6) edge (v7)
				(v7) edge (v8)
				(v8) edge (v1)
				(x) edge (v2)
				(x) edge (v4)
				(x) edge (v6)
				(x) edge (v8)

				;
			\end{tikzpicture}
			\caption{$H_4$ with a maximum gp-set}
			\label{fig:H4}
		\end{figure}
		
		We will show that for $r \geq 4$ the half-wheel has $\mono(H_r) = 2$ and $\gp(G) = r$. First let $M$ be an mp-set in $H_r$.  $M$ can contain at most two vertices of the induced cycle $C_{2r}$, so $\mono (H_r) \leq 3$ and equality holds only if $M$ contains $x$ and two vertices $i,j \in V(C_{2r})$. If $i$ and $j$ are both even, then $i, x, j$ is a monophonic path passing through three vertices of $M$, so we can take $j$ to be odd. Suppose that $i$ is even and $j$ odd; if $i \sim j$, then $j,i,x$ is a monophonic path containing three vertices of $M$, whereas if $i \not \sim j$, then $i,x,j-1,j$ is the required path. Finally if $i$ and $j$ are both odd, then $i,i-1,x,j+1,j$ is the required induced path, where we assume that $j = i+2$ if $d(i,j) = 2$. Thus $\mono(H_r) = 2$.
		
		The set of all even integers on $C_{2r}$ is a general position set, so $\gp(H_r) \geq r$. Suppose that there is a general position set $K$ in $H_r$ with $r+1$ elements.  Suppose that $x \not \in K$. Then $K$ contains two vertices that are neighbours on $C_{2r}$, without loss of generality $i,i+1 \in K$, where $i$ is even.  Then the distance from $i+1$ to any even vertex $j$ of $C_{2r}$ apart from $i$ and $i+2$ is three and $i+1, i, x, j$ is a geodesic. Also $i+2 \not \in K$, as $i, i+1, i+2$ is a geodesic. Therefore $K$ must consist of the set of odd vertices of $C_{2r}$ together with $\{ i\} $. However $i-1,i,i+1$ would then be a geodesic containing three vertices of $K$.
		
		Hence we can assume that $x \in K$. Suppose that $K$ contains an even vertex $i$ of $C_{2r}$.  Between any two even vertices of $C_{2r}$ there is a geodesic passing through $x$, so $K$ can contain no other even vertex of $C_{2r}$. Also $K$ cannot contain $i+1$ or $i-1$, as $i$ is contained in $(i-1),x$- and $(i+1),x$-geodesics.  Therefore $K$ would contain at most $2+(r-2) = r$ vertices, a contradiction. Hence $K$ consists of $x$ and the odd vertices of $C_{2r}$; however, $1, 2, x, 4, 5$ is a geodesic containing three vertices of $K$. It follows that for $r \geq 4$ we have $\gp(H_r) = r$. 
		
		We now turn to the case of larger $a$; let $3 \leq a < b$. Consider the graph $H_{b,a}$ formed by attaching $a-2$ leaves to the central vertex of the half-wheel $H_{b-a+2}$. The half-wheel is an isometric subgraph of $H_{b,a}$, so it follows from the previous discussion that if $b-a+2 \geq 4$ (i.e. if $b \geq a+2$) any largest monophonic and general position sets of $H_{b,a}$ can contain at most two or $b-a+2$ vertices of $H_{b-a+2}$ vertices respectively, so that $\mono (H_{b,a}) \leq (a-2)+2 = a$ and $\gp (H_{b,a}) \leq (b-a+2)+(a-2) = b$. If $L$ is the set of $a-2$ leaves of $H_{b,a}$, then the sets $\{ 2,4\} \cup L$ and $\{ 2,4,6,\dots, 2(b-a+2)\} \cup L$ are in monophonic and general position respectively, implying that $\mono (H_{b,a}) = a$ and $\gp (H_{b,a}) = b$. This leaves only the case $b = a+1$ to examine. It follows as above that for $b = a+1$ we have $\mono (H_{b-a+2}) = \mono (H_3) = 2$, so that $\mono (H_{a+1,a}) = a$. It is easily seen that $\gp (H_3) = 4$, but the only $\gp $-sets of $H_3$ contain the central vertex $x$. By Lemma~\ref{lemma:cut vertex} we can assume that a $\gp $-set of $H_{a+1,a}$ does not contain $x$, so that $\gp (H_{a+1,a}) \leq (a-2) + 3 = a+1$ and the set $\{ 2,4,6\} \cup L$ shows that we have equality.
	\end{proof}
	Theorem~\ref{half wheel} suggests the following problem for further research.
	\begin{problem}
		For given $2 \leq a \leq b$, for which values of $n$ does there exist a graph $G$ with order $n$, $\mono (G) = a$ and $\gp (G) = b$?
	\end{problem}

	The \emph{independent position number} of a graph $G$, which we will denote by $\igp (G)$, is the number of vertices in a largest set $S \subseteq V(G)$ that is both independent and in general position. This parameter was studied in~\cite{ThoCha}. If we define the \emph{independent monophonic position number} $\imp (G)$ to be the number of vertices in a maximum subset $S \subseteq V(G)$ such that $S$ is independent and in monophonic position, then we have both $\imp (G) \leq \mono(G) \leq \gp(G)$ and $\imp(G) \leq \igp(G) \leq \gp(G)$. This raises the question of whether there is a relationship between $\mono(G)$ and $\igp (G)$. We can easily answer this in the negative.
	
	\begin{theorem}\label{thm:igp vs. mp}
		There exists a graph $G$ with $\igp (G) = a$ and $\mono(G) = b$ if and only if $1 = a \leq b$ or $2 \leq a,b$.
	\end{theorem}
	\begin{proof}
		The only graph with $\mono $-number one is $K_1$, which implies the necessity of the conditions. For $n \geq 1$ we have $\igp (K_n) = 1$ and $\mono (K_n) = n$, so we can assume that $2 \leq a,b$. For the case $a \leq b$, for $r \geq 0$ and $s \geq 1$, define $R(r,s)$ to be the graph formed by attaching $r$ leaves to one vertex of a copy of $K_{s+1}$; then we have $\mono (R(r,s)) = r+s$ and $\igp (R(r,s)) = r+1$, so that the graph $R(a-1,b-a+1)$ has the required properties.
		
		Consider now the case $a > b$. Let the graph $K_{s,s}^-$ be the complete bipartite graph $K_{s,s}$ with a perfect matching deleted. Form the graph $P(r,s)$ by adding an extra vertex $x$ to $K_{s,s}^-$ and joining $x$ by an edge to each vertex of one of the partite sets of $K_{s,s}^-$, then adding $r$ leaves to $x$. This graph satisfies $\mono (P(r,s)) = r+2$ and $\igp (P(r,s)) = r+s$. Hence the graph $P(b-2,a-b+2)$ has the necessary parameters.
	\end{proof}
	The graph $R(r,s)$ from the proof of Theorem~\ref{thm:igp vs. mp} can also be used to answer Question 2 from~\cite{KlaRalYer}. The \emph{dissociation number} $\diss(G)$ of a graph $G$ is the number of vertices in a largest subset $S \subseteq V(G)$ such that $G[S]$ has maximum degree at most one. Recall also that the \emph{$2$-position-number} $\gptwo (G)$ is the number of vertices in a largest set $K \subset V(G)$ such that no shortest path in $G$ of length two contains three vertices of $K$~\cite{KlaRalYer}. It is easily seen that any dissociation set is in $2$-general position, so that $\diss (G) \leq \gptwo (G)$. Furthermore for $2 \leq a \leq b$ we have $\diss (R(a-2,b-a+2)) = a$ and $\gptwo (R(a-2,b-a+2)) = b$ if $a \geq 3$, whilst if $a = 2$ the clique $K_b$ has the required parameters.

	For two vertices $u,v$ of a graph $G$ the set $I[u,v]$ consists of all vertices that lie on a $u,v$-geodesic; for a subgraph $S \subseteq V(G)$ we have $I[S] = \bigcup_{u,v \in S}I[u,v]$. If $I[S] = S$, then $S$ is \emph{convex}. A smallest convex set containing a given subset $S$ is the \emph{convex hull} $I_h[S]$ of $S$; if $I_h[S] = V(G)$, then $S$ is a \emph{hull-set} and the number of vertices in a smallest hull set is the \emph{hull number} $h(G)$ of $G$. The hull number has been widely studied, for example in~\cite{Araujo-2013,Cagaanan-2004,Dourado-2009}. In~\cite{ullas-2016} it was shown that for any $a,b \in \mathbb{N}$ such that $2 \leq a \leq b$ there exists a graph $G$ with $h(G) = a$ and $\gp(G) = b$. The corresponding parameter for monophonic paths, the \emph{monophonic hull number} $h_m(G)$ (discussed in Section~\ref{introduction}), has been studied in~\cite{Dourado-2008,Dourado-2010,Ignacio-2003,pelayo-2013}. We now prove a realisation theorem for the monophonic hull number and the monophonic position number. We will use the following two results in our analysis.
	
	\begin{lemma}[\cite{pelayo-2013}] 
		\label{theorem:simplicial vertex} 
		Every monophonic hull set of a graph $G$ contains all of its simplicial vertices. 
	\end{lemma}
	
	\begin{observation}\label{observation:hull}
		Let $M$ be a minimum monophonic hull set of a connected graph $G$ and let $a,b \in M$. If a vertex $c \not = a,b$ lies on an $a,b$-monophonic path, then $c \not \in M$. 
	\end{observation}
	It follows from Observation~\ref{observation:hull} that every minimum monophonic hull set of a connected graph $G$ is an mp-set in $G$. Consequently, for any graph with order $n \geq 2$ we have $2 \leq h_m(G) \leq \mono(G) \leq n$.  In view of this inequality, we have the following realisation result.
	
	\begin{theorem}
		For all $n,a,b \geq 1$ there exists a connected graph $G$ with $h_{m}(G) = a$, $\mono (G) = b$ and order $n$ if and only if $a = b = n$ or $2 \leq a \leq b \leq n-1$.   
	\end{theorem}
	\begin{proof}
		The only graph with $h_m(G) = 1$ or $\mono (G) = 1$ is $K_1$, so assume that $a \geq 2$. If $b = n$, then $G$ must be a clique, so that $a = b = n$. It remains only to prove the existence of a graph $G$ with order $n$, $h_m(G) = a$ and $\mono(G) = b$ for all $2 \leq a \leq b \leq n-1$. 
		
		Let $W$ be a clique of order $b$ and $P_{\ell }$ be a path of order $\ell \geq 1$ with vertices $x_1,\dots ,x_{\ell }$ such that $x_i \sim x_{i+1}$ for $1 \leq i \leq \ell-1$. For $2 \leq a \leq b$ let $G(a,b,\ell)$ be the graph formed from $W$ and $P_{\ell}$ by joining $x_1$ to $b-a+1$ vertices of $W$ by edges. Write $X = V(W) \setminus N(x_0)$ and $Y = V(W) \cap N(x_0)$, so that $|X| = a-1$ and $|Y| = b-a+1$. The order of $G(a,b,\ell)$ is $b+\ell $. By Lemma~\ref{theorem:simplicial vertex}, any monophonic hull set of $G(a,b,n)$ must contain $X \cup \{ x_{\ell }\}$; conversely, this subset is a monophonic hull set, so $h_m(G(a,b,\ell )) = a$. As the clique number of $G(a,b,\ell )$ is $b$, we have $\mono (G(a,b,\ell )) \geq b$. For the converse, let $M$ be a maximum mp-set of $G(a,b,\ell)$. $M$ contains at most two vertices of $P_{\ell }$ and if $|M \cap V(P_{\ell })| = 2$, then $M \cap W = \emptyset $ and $|M| \leq b$. If $|M \cap V(P_{\ell })| = 1$, then $M$ cannot contain vertices of both $X$ and $Y$, so that $|M| \leq \max \{ a,b-a+2\} \leq b$. Thus $|M| = b$ and $\mono (G(a,b,\ell)) = b$. Therefore the graph $G(a,b,n-b)$ has the required properties.  
	\end{proof}
	
	\section{Computational complexity}\label{sec:comp}
	
	In this section, we study the computational complexity of the problem of finding the monophonic position number of a general graph. To this end, we formally define the decision version of the problem:
	\begin{definition}\label{def:prob}
		{\sc Monophonic position set} \\
		{\sc Instance}: A graph $G$, a positive integer $k\leq |V(G)|$. \\
		{\sc Question}: Is there a monophonic position set $S$ for $G$ such that $|S|\geq k$?
	\end{definition}
	
	The problem is hard to solve as shown by the next theorem.
	
	\begin{theorem}\label{thm:NP}
		The {\sc Monophonic position set} problem is NP-hard.   
	\end{theorem}
	
	\begin{proof}
		
		We prove that the {\sc Clique} problem polynomially reduces to {\sc Monophonic position set}. 
		An instance of {\sc Clique} is given by  a graph $G$ and a positive integer $k\leq |V|$.
		The  {\sc Clique} problem asks whether $G$ contains a clique $K\subseteq V(G)$ of order $k$ or more.

		The NP-completeness of {\sc Clique} is well known, as it is one of the original list of 21 NP-complete problems presented in~\cite{Karp72}. We polynomially transform an instance $(G,k)$ of {\sc Clique} to an instance $(G',k')$ of {\sc Monophonic position set} so that $G$ has a clique of order $k$ or more if and only if $G'$ has a monophonic position set of order $k'$ or more. 
		
		Given an instance $(G,k)$, the graph $G'$ is built as follows:
		\begin{align*} 
			V(G')&=\{v',v''~|~v\in V(G)\}\\
			E(G')&=\{u' v'~|~uv\in E(G)\}\cup\{u'v''~|~u,v\in V(G)\}\cup\{u'' v''~|~u,v\in V(G)\}
		\end{align*}
		
		In words, $G'$ contains a subgraph $H$ isomorphic to $G$ and a clique graph $H'$ of order $n=|V(G)|$ such that $G'=H \vee H'$. As for $k'$, we set $k'=n+k$. 
		
		As a preliminary results, note that $\omega(G')=\omega(G)+n$. Indeed, if $S$ is any clique in $G$, then $S'=\{s'~|~s\in S\}$ is a clique in $H$, and  $S'\cup V(H')$ is a clique in $G'$, since any vertex in $H'$ is a universal vertex of $G'$. Then $\omega(G')\geq\omega(G)+n$. 
		On the contrary, if $S$ is any clique of $G'$, then $|S\cap V(H)|\leq \omega(G)$ and $|S\cap V(H')|\leq n$. Hence $\omega(G')\leq\omega(G)+n$.
		
		By Proposition~\ref{lemma:join},
		$\mono(G')=\mono(H \vee H') = \max\{\omega (H) + \omega (H'), \mono(H), \mono(H') \}$. We have that  $\omega (H')=\mono(H')=n$ for $H'$ is a clique; $\mono(H) = \mono(G)$ and $\omega (H)=\omega(G)\leq n$ for $H$ is isomorphic to $G$. Then $\mono(G')= \max\{\omega (H) + n, \mono(H), n \}= \omega (H) + n$, being $\mono(H)\leq n$.
		
		Now assume that an instance $(G,k)$ of {\sc Clique} has a positive answer, then $\omega(G)=\omega(H)\geq k$. Then $\mono(G')= \omega (H) + n\geq k + n = k'$, and hence {\sc Monophonic position set} has a positive answer. On the other hand, if the instance $(G,k)$ of {\sc Clique} has a negative answer then $\omega(G) =\omega(H) < k$, which in turn implies that {\sc Monophonic position set} has a negative answer, since $\mono(G')= \omega (H) + n < k +n$.
	\end{proof}
	
	By the proof of Theorem~\ref{thm:NP}, it is not clear if {\sc Monophonic position set} is NP-complete. However if we restrict the problem to instances $(G,k)$ with $k>|V(G)|/2$ and $G=H \vee K$, where $H$ is a generic graph and $K$ is a clique graph having the same order of $H$, the problem is NP-complete. Indeed, given a solution, it can be tested in polynomial time if it is a clique and if its order is larger than $k$.

	\section*{Acknowledgements}
	The first author thanks the University of Kerala for providing JRF. The third author gratefully acknowledges funding support from EPSRC grant EP/W522338/1 and London Mathematical Society grant ECF-2021-27 and thanks the Open University for an extension of funding during lockdown. The authors thank Dr. Erskine for help with computation of the mp-numbers of cubic cages and the anonymous reviewers for their useful feedback.



\begin{thebibliography}{99}
		\bibliographystyle{plain}
		
		\bibitem{AkbHorWan} Akbari, S., Horsley, D. \& Wanless, I.M., \emph{Induced path factors of regular graphs.} J. Graph Theory 97 (2) (2021), 260-280.
		
		\bibitem{O1} Anand, B.S., Chandran S.V., U., Changat, M., Klav\v{z}ar, S. \& Thomas E.J., \emph{Characterization of general position sets and its applications to cographs and bipartite graphs.} Appl. Math. Comput. 359 (2019), 84-89.
		
		\bibitem{Araujo-2013} Araujo, J., Campos, V., Giroire, F., Nisse, N., Sampaio, L. \& Soares, R., \emph{On the hull number of some graph classes.} Theor. Comput. Sci. 475 (2013), 1-12. 
		
		\bibitem{BonMur} Bondy, J.A. \& Murty, U.S.R., \emph{Graph Theory with Applications}  London: Macmillan, Vol. 290 (1976).
		
		\bibitem{Cagaanan-2004}
		Cagaanan, G.B. \& Canoy Jr., S.R., \emph{On the hull sets and hull number of the Cartesian product of graphs.} Discrete Math. 287 (1-3) (2004), 141-144.
		
		\bibitem{ullas-2016} 
		Chandran S.V., U. \& Parthasarathy, G.J., \emph{The geodesic irredundant sets in graphs.} Int. J. Math. Comb. 4 (2016), 135-143.
		
		\bibitem{JDA04} Cicerone, S. \& Di Stefano, G., \emph{Networks with small stretch number.} J. Discrete Algorithms 2 (4) (2004), 383-405.
		
		\bibitem{ChaMcCSheHosHas} Chartrand, G., McCanna, J., Sherwani, N., Hossain, M. \& Hashmi, J., \emph{The induced path number of bipartite graphs.} Ars Comb. 37 (1994), 191-208.
		
		\bibitem{CD06} Cornelsen, S. \& Di Stefano, G. \emph{Treelike comparability graphs.} Discrete Appl. Math. 157 (8) (2009), 1711-1722.
		
		\bibitem{GDS12} Di Stefano, G. \emph{Distance-hereditary comparability graphs.} Discrete Appl. Math. 160 (18) (2012), 2669-2680.
		
		\bibitem{diStef} Di Stefano, G., \emph{Mutual visibility in graphs.} Appl. Math. Comput. 419 (2022), 126850. 
		
		\bibitem{Dourado-2008} Dourado, M.C., Protti, F. \& Szwarcfiter, J.L., \emph{Algorithmic aspects of monophonic convexity.} Electron. Notes Discrete Math. 30 (2008), 177-182.
		
		\bibitem{Dourado-2009} Dourado, M.C., Gimbel, J.G., Kratochv'{i}l, J., Protti, F. \& Szwarcfiter, J.L., \emph{On the computation of the hull number of a graph.} Discrete Math. 309 (18) (2009), 5668-5674.
		
		\bibitem{Dourado-2010} Dourado, M.C., Protti, F. \& Szwarcfiter, J.L., \emph{Complexity results related to monophonic convexity.} Discrete Appl. Math. 158 (12) (2010), 1268-1274. 
		
		\bibitem{dudeney-1917} Dudeney, H.E., \emph{Amusements in mathematics.} Nelson, Edinburgh (1917).
		
		\bibitem{Erskine} Erskine, G., personal communication (2020)
		
		\bibitem{Ghorbani-2020} Ghorbani, M., Maimani, H.R., Momeni, M., Mahid, F.R., Klav\v{z}ar, S. \& Rus, G., \emph{The general position problem on Kneser graphs and on some graph operations.} Discuss. Math. Graph T. 41 (4) (2019), 1199-1213. 
		
		\bibitem{Hall} Hall, P., \emph{On representatives of subsets.} J. London Math. Soc., 10.1 (1935), 26-30.
		
		\bibitem{Howorka-77} Howorka, E., \emph{A characterization of distance-hereditary graphs.}
		Q. J. Math. 28 (4) (1977), 417-420.
		
		\bibitem{Howorka-81} Howorka, E., \emph{A characterization of Ptolemaic graphs.} J. Graph Theory 5 (3) (1981), 323-331.
		
		\bibitem{Ignacio-2003} Hernando, C., Mora, M., Pelayo, I. \& Seara, C., \emph{On monophonic sets in graphs.} Discrete Math., submitted (2003).
		
		\bibitem{Karp72} Karp, R.M., \emph{Reducibility among combinatorial problems.} In Complexity of Computer Computations (1972), 85-103.
		
		\bibitem{KlaKuzPetYer} Klav\v zar, S., Kuziak, D., Peterin. I. \& Yero, I.G., \emph{A Steiner general position problem in graph theory.} Comput. Appl. Math. 40 (6) (2021), 1-15.
		
		\bibitem{KlaRalYer} Klav\v zar, S., Rall, D.F. \& Yero, I.G., \emph{General $d$-position sets.} Ars Math. Contemp. 21 (1) (2021), 1-03.
		
		\bibitem{Klavzar-2019} Klav\v{z}ar, S. \& Yero, I.G., \emph{The general position problem and strong resolving graphs.} Open Math. 17 (1) (2019), 1126-1135.
		
		\bibitem{Korner} K\"orner, J. \emph{On the extremal combinatorics of the Hamming space.} J. Comb. Theory Ser. A 71 (1) (1995), 112-126.
		
		\bibitem{manuel-2018a} Manuel, P. \& Klav{\v z}ar, S., \emph{A general position problem in graph theory.} Bull. Aust. Math. Soc. 98 (2) (2018), 177-187.
		
		\bibitem{manuel-2018b} Manuel, P. \& Klav{\v z}ar, S. \emph{The graph theory general position problem on some interconnection networks.} Fundam. Inform. 163 (4) (2018), 339-350.
		
		\bibitem{Neethu-2020} Neethu, P.K., Chandran S.V., U., Changat, M. \& Klav{\v z}ar, S., \emph{On the general position number of complementary prisms.} Fundam. Inform. 178 (3) (2021), 267-281.
		
		\bibitem{Patkos-2019} Patk\'{o}s, B., \emph{On the general position problem on Kneser graphs.} Ars Math. Contemp. 18 (2) (2020), 273-280. 
		\bibitem{pelayo-2013} Pelayo, I.M., \emph{Geodesic convexity in graphs.} New York: Springer (2013).
		
		\bibitem{Thomas-2020} Thomas, E.J. \& Chandran S.V., U., \emph{Characterization of classes of graphs with large general position number.} AKCE Int. J. Graphs Comb. (2020), 1-5.
		
		\bibitem{ThoCha} Thomas, E.J. \& Chandran S.V., U., \emph{On independent position sets in graphs.} Proyecciones (Antofagasta) 40 (2) (2021), 385-398.
		
	\end{thebibliography}
\end{document}